\documentclass[11pt]{article}
\usepackage{graphicx}
\usepackage{amssymb}
\usepackage{amsmath}
\usepackage{amsthm}
\usepackage{amsfonts}
\usepackage{url}
\usepackage{hyperref}
\usepackage{breakurl}
\newtheorem{thm}{Theorem}

\begin{document}
\bibliographystyle{plain}
\title{~\\[-40pt] 
Finding D-optimal designs by randomised decomposition and switching}

\author{Richard P. Brent\\
Australian National University}
\date{Dedicated to Kathy Horadam\\ on the occasion
	of her sixtieth birthday}
\maketitle
\thispagestyle{empty}                   % To avoid page number

\begin{abstract}
A square $\{+1,-1\}$-matrix of order $n$ 
with maximal determinant is called a {\em saturated D-optimal design}.
We consider some cases of saturated D-optimal designs
where $n>2$, $n \not\equiv 0 \bmod 4$, so the Hadamard bound
is not attainable, but bounds due to Barba 
or Ehlich and Wojtas may be attainable. If $R$ is a matrix with
maximal (or conjectured maximal) determinant, then $G = RR^T$ 
is the corresponding
{\em Gram matrix}.  For the cases that we consider, maximal or conjectured
maximal Gram matrices are known. We show how to generate many
Hadamard equivalence classes of solutions from a given Gram matrix $G$, using
a randomised decomposition algorithm and row/column switching. 
In particular, we consider orders $26$, $27$ and~$33$,
and obtain new saturated D-optimal designs (for order $26$) and new
conjectured saturated D-optimal designs (for orders $27$ and~$33$).
\end{abstract}

\section{Introduction}		\label{sec:intro}

The {\em Hadamard maximal determinant} (maxdet) problem is to find the
maximum determinant $D(n)$ of a square $\{+1,-1\}$-matrix of given order~$n$.
Such a matrix $A$ with maximal $|\det(A)|$ is called a
{\em saturated D-optimal design} of order~$n$.
We are only concerned with the absolute value of the determinant,
as the sign may be changed by a row or column interchange.
 
Hadamard~\cite{Hadamard93} 
showed that $D(n) \le n^{n/2}$, and this bound is attainable 
for $n > 2$ only if $n \equiv 0 \bmod 4$.
The \emph{``Hadamard conjecture''} (due to Paley~\cite{Paley33}) 
is that Hadamard's bound is
attainable for all $n \equiv 0 \bmod 4$.
In this paper we are concerned with
``non-Hadamard'' cases $n > 2$, $n \not\equiv 0 \bmod 4$.
For such orders the Hadamard bound is not attainable, but other upper
bounds due to Barba~\cite{Barba33}, Ehlich~\cite{Ehlich64a,Ehlich64b}
and Wojtas~\cite{Wojtas64} may be attainable.
For lower bounds on $D(n)$, see Brent and Osborn~\cite{rpb249}, and
the references given there.

We say that two $\{+1,-1\}$ matrices $A$ and $B$ are
\emph{Hadamard-equivalent} (abbreviated H-equivalent) if $B$ can be obtained
from $A$ by a signed permutation of rows and/or columns.
If $A$ is H-equivalent to $B$ or to $B^T$ then we say that
$A$ and $B$ are \emph{extended Hadamard-equivalent}
(abbreviated HT-equivalent).
Note that, if $A$ is HT-equivalent to $B$, then
$|\det(A)| = |\det(B)|$.

If $A$ is H-equivalent to $A^T$ then we say that $A$ is {\em self-dual}.
We say that a Hadamard equivalence class is self-dual if the class
contains a self-dual matrix (equivalently, if the duals\footnote{We
use ``dual'' and ``transpose'' interchangeably.} of all matrices
in the class are also in the class).

If we know (or conjecture) $D(n)$, it is of interest to find all (or most)
Hadamard equivalence classes
of $\{+1,-1\}$ matrices with determinant $\pm D(n)$. In this paper we 
consider the orders $26$, $27$ and $33$; similar methods
can be used for certain other orders.

In \S\ref{sec:decomp} we consider the randomised decomposition of
(candidate) Gram matrices.  Then, in \S\ref{sec:switching}, we show how one
solution can often be used to generate other, generally not Hadamard
equivalent, solutions via switching. The graph of Hadamard equivalence
classes induced by switching is defined in \S\ref{sec:switching}. We
conclude with some new results for the orders $26$, $27$ and $33$
in \S\S\ref{sec:order26}--\ref{sec:order33}.

\subsection*{Upper bounds}
A bound which holds for all odd orders, and which is known to be
sharp for an infinite sequence of orders $\equiv 1 \pmod 4$, is 
\begin{equation} \label{eq:Barbabound}
D(n) \le (2n-1)^{1/2}(n-1)^{(n-1)/2},
\end{equation}
due independently to 
Barba~\cite{Barba33}
and Ehlich~\cite{Ehlich64a}.
We call it the {\em Barba} bound.
Brouwer~\cite{Brouwer83} showed that the Barba bound (\ref{eq:Barbabound}) 
is sharp if
$n = q^2 + (q+1)^2$ for $q$ an odd prime power. 
The bound is also sharp in some other
cases, e.g.~$q=2$ and $q=4$. It is not achievable unless $n$ is the sum
of two consecutive squares.

An upper bound due to Ehlich~\cite{Ehlich64b}
applies only in the case $n \equiv 3 \pmod 4$.
We refer to~\cite{rpb244,Ehlich64b,Orrick-www,Osborn02}
for details of this bound, which is rather complicated.
The Ehlich bound is not known to be sharp for any order $n > 3$.

Another bound, 
\begin{equation}			\label{eq:EWbound}
D(n) \le (2n-2)(n-2)^{(n-2)/2},
\end{equation}
due to Ehlich~\cite{Ehlich64a} and Wojtas~\cite{Wojtas64}, 
applies in the case $n \equiv 2 \pmod 4$.
It is known to be sharp in the following cases:
(A) $n = 2(q^2 + (q+1)^2)$, 
where $q$ is an odd prime power (see Whiteman~\cite{Whiteman90}); and
(B) $n = 2(q^2 + q + 1)$, where $q$ is any (even or odd) 
prime power~\cite{Spence75,KKS91}.

\subsection*{Gram matrices}
If $R$ is a given square matrix then the symmetric matrix $G = RR^T$ is
called the {\em Gram matrix} of $R$. 
We may also consider the {\em dual Gram matrix} $H=R^TR$.
Since $\det(G) = \det(R)^2$, the bounds mentioned above on $\det(R)$ 
are equivalent to bounds on $\det(G)$. Indeed, this observation explains
the form of the bounds.  For example, the Barba bound corresponds
to a matrix $G = (g_{i,j})$ given by 
$g_{i,j} = n$ if $i=j$ and $g_{i,j} = 1$ if $i \ne j$.
It is easy to show via a well-known rank-1 update formula 
that
$\det(G) = (2n-1)(n-1)^{n-1}$.

Given a symmetric matrix $G$ with suitable determinant, we say $G$ is a
{\em candidate Gram matrix}.  It is the Gram matrix of a
$\{+1,-1\}$ matrix if and only if it
decomposes into a product of the form $G = RR^T$, where
$R$ is a square $\{+1,-1\}$ matrix.

\section{Decomposition of candidate Gram matrices} \label{sec:decomp}

Suppose that a (candidate) Gram matrix $G$ of order $n$ is known.  
We want to find one or more
$\{+1,-1\}$ matrices $R$ such that $G = RR^T$.
Let the rows of $R$ be $r_1^T, \ldots, r_n^T$.
Then
\[r_i^Tr_j = g_{i,j},\; 1 \le i, j \le n.\]

If we already know the first $k$ rows, then we get $k$ 
{\em single-Gram}
constraints involving row $k+1$:
\[r_i^Tr_{k+1} = g_{i,k+1}\; {\rm for}\; 1 \le i \le k.\]
These are linear constraints in the unknowns $r_{k+1}$.
We may be able to find one or more solutions for row $k+1$ satisfying
the single-Gram constraints, or there may be no solutions, in which
case we have to backtrack.

Our algorithm is described in~\cite[\S4]{rpb244}, so we omit
the details here. We merely note that it is possible to take advantage
of various symmetries to reduce the size of the search space, and that
it is possible to prune the search using {\em gram-pair} constraints
of the form $G^{j+1} = RH^jR^T$ ($j > 0$)
if we know the dual Gram matrix $H = R^TR$.  In the cases considered
below, there is (up to signed permutations) only one candidate Gram
matrix with the required determinant, so there is no loss of generality
in assuming that $G = H$.

The search can be regarded as searching a (large) tree with (at most)
$n$ levels, where each level corresponds to a row of $R$.  A deterministic
search typically searches the tree in depth-first fashion~-- at each node,
recursively search the subtrees defined by the children of that node.
The aim is to find one or more leaves at the $n$-th level of the tree,
since these leaves correspond to complete solutions~$R$.

Deterministic, depth-first search may take a
long time searching fruitlessly for solutions in subtrees where no
solutions exist. When $G$ is decomposable, but difficult to decompose
using a deterministic search, we may be able to do better with a
randomised search.

In the randomised search, at each node we randomly choose a small number of
children and recursively search the subtrees defined by these children. A
good choice of the average number of children chosen per node, say $\mu$,
can be determined experimentally. Too small a value makes it unlikely that
a solution will be found; too large a value makes the search take too long.
We found empirically that $\mu \approx 1.3$ works well in the cases
considered below. Thus, at each node traversed in the search we choose one
child (if there are any) and, with probability about $\mu-1 \approx 0.3$,
also choose a second child (if there is one), then recursively search
the subtrees defined by the selected children.

For example, in the case $n=27$, there is a known Gram matrix $G$, due to
Tamura~\cite{Tamura05}, which decomposes into $RR^T$, giving a $\{+1,-1\}$
matrix $R$ of determinant
$546\times 6^{11}\times 2^{26}$. 
This determinant is conjectured to be maximal.

A deterministic search fails to decompose Tamura's $G$ 
in $24$ hours (exploring over $10^8$ nodes but reaching only
depth $17$ in the search tree).  
The tree size is probably greater than $4 \times 10^9$.

On the other hand, our randomised search routinely finds a decomposition of
$G$ in about $90$ seconds. In this way we have found many different
H-classes of solutions. Further details are given in \S\ref{sec:order27}.

\subsection*{Nonuniformity of sampling}

Unfortunately, the randomised search strategy described above does not
guarantee that the set of H-classes of solutions 
(or the set of all $\{+1,-1\}$ matrices of the given order and determinant) 
is sampled uniformly. There
are two reasons for lack of uniformity.  First, the tree-generation
algorithm introduces non-uniformity by taking advantage of symmetries to
reduce the size of the tree.  Second, the number of matrices in a class is
inversely proportional to the order of the automorphism group of the class,
so even if all the $\{+1,-1\}$ matrices were sampled
uniformly, the H-classes would not necessarily be sampled 
uniformly\footnote{To a certain extent, these two sources of bias may tend 
to cancel.}.

\section{Switching}	\label{sec:switching}
\emph{Switching} is an operation on square $\{+1,-1\}$ matrices 
which preserves the absolute value of the determinant 
but does not generally preserve Hadamard
equivalence or extended Hadamard equivalence.

Thus, switching can be
used to generate many inequivalent maxdet solutions from one solution.
This idea was introduced by Denniston~\cite{Denniston82}
and used to good effect by Orrick~\cite{Orrick08a}.
Similar ideas have been used by Wanless~\cite{Wanless04} and others
in the context of latin squares.

We only consider {switching a closed quadruple} of rows/columns.
There are other possibilities, e.g.\ switching Hall sets~\cite{Orrick08a}.

\subsection*{Switching a closed quadruple of rows/columns}

Suppose that a $\{+1,-1\}$ matrix $R$ is H-equivalent to a matrix having
a {\em closed quadruple} of rows, i.e.\ four rows of the form%
\footnote{We write ``$+$'' for $+1$ and ``$-$'' for $-1$.}:
\[
\left [
\begin{matrix} 
+ \cdots +& - \cdots - & - \cdots - & + \cdots + \\
+ \cdots +& - \cdots - & + \cdots + & - \cdots - \\
+ \cdots +& + \cdots + & - \cdots - & - \cdots - \\
+ \cdots +& + \cdots + & + \cdots + & + \cdots + \\
\end{matrix}
\right ]
\]
Then \emph{row switching} flips the sign of the leftmost block, giving
\[
\left [
\begin{matrix} 
- \cdots -& - \cdots - & - \cdots - & + \cdots + \\
- \cdots -& - \cdots - & + \cdots + & - \cdots - \\
- \cdots -& + \cdots + & - \cdots - & - \cdots - \\
- \cdots -& + \cdots + & + \cdots + & + \cdots + \\
\end{matrix}
\right ]
\]

This is H-equivalent to flipping
the signs of all but the leftmost block, which has a
nicer interpretation in terms of switching edges in the corresponding
bipartite graph~\cite{McKay79,McKay09}.

It is easy to see that row switching preserves the inner products of
each pair of columns of R, so preserves the dual Gram matrix $R^TR$, and
hence preserves $|\det(R)|$.  However, it does not generally preserve
H-equivalence or HT-equivalence.

\emph{Column switching} is dual to row switching~-- 
instead of a closed quadruple of
four rows, it requires a closed quadruple of four columns.

\subsection*{Equivalence classes generated by switching, and their graphs}

Let $\cal A$ and $\cal B$ be two H-equivalence classes of matrices.
We say that $\cal A$ and $\cal B$ are {\em switching-equivalent}
(abbreviated ``S-equivalent'') if there exists $A \in \cal A$
and $B \in \cal B$ such that $A$ 
can be transformed to $B$ by a sequence of
row and/or column switching 
operations\footnote{S-equivalence is the same as
Orrick's {\em Q-equivalence}~\cite{Orrick08a,Orrick08b}
in the cases that we consider, but the concepts are different
if $n \equiv 4 \bmod 8$.}.
The size of an S-equivalence class $\cal C$, denoted by $||{\cal C}||_S$,
is the number of H-equivalence classes that it contains.

If the H-equivalence classes corresponding to matrices $A$ and $B$ 
are in the same S-equivalence
class, then we write $A \leftrightarrow B$. Thus, this notation means that
there is a sequence of row/column switches that transforms $A$ to 
a matrix H-equivalent to $B$. We say that $A$ is {\em S-equivalent} to $B$.

If $\cal A$ and $\cal B$ are two HT-equivalence classes of matrices, then we
say that $\cal A$ and $\cal B$ are {\em ST-equivalent} if there exists 
$A \in \cal A$ and $B \in \cal B$ such that $A$ can be transformed to $B$ by a
sequence of row and/or column switching operations\footnote{Thus,
for all $\alpha \in \cal A$ and $\beta \in \cal B$, we have
{\em either} $\alpha \leftrightarrow \beta$ or $\alpha \leftrightarrow
\beta^T$.}.
The size of an
ST-equivalence class $\cal C$, denoted by $||{\cal C}||_{ST}$, is the number
of HT-equivalence classes that it contains.

We say that two {\em matrices} $A$ and $B$ are ST-equivalent
if the corresponding HT-classes ${\cal A} \ni A$ and
${\cal B} \ni B$ are ST-equivalent.
Thus, two matrices $A$ and $B$ are ST-equivalent
if a matrix H-equivalent to $B$ 
can be obtained from $A$ by a sequence
of row switches, column switches and/or transpositions.

The {\em weight} $w(H)$ of a matrix $H$ 
(or of the Hadamard class ${\cal H} \ni H$) 
is defined by
\begin{equation}
w(H) = w({\cal H}) = \frac{1}{|{\rm Aut}(H)|}\,,
\end{equation}
where ${\rm Aut}(H)$ is the automorphism group
of $H$.
The {\em weight} 
$w({\cal C})$ of an S-class $\cal C$ is defined by
\begin{equation}
w({\cal C}) = \sum_{{\cal H} \in {\cal C}} w({\cal H})\,. \label{eq:weight}
\end{equation}
The probability of finding a class by uniform random sampling of $\{+1,-1\}$
matrices is proportional to the weight of the class, so the classes with
smallest weight are in some sense the hardest to find. (However, as 
observed at the end of \S\ref{sec:decomp}, we do not sample uniformly.)

Associated with an S-equivalence class ${\cal S} = \{H_1, \ldots, H_s\}$
of size $s$ there is a graph\footnote{We ignore any loops or multiple
edges, so all graphs considered here are simple.} 
$G = G({\cal S})$ whose vertices are
the H-classes $H_1, \ldots, H_s$ contained in ${\cal S}$, and where
an edge connects two distinct vertices $H_i, H_j$ if a matrix in $H_i$ can
be transformed to a matrix in $H_j$ by a single row/column switching
operation.  Similarly, for an ST-equivalence class ${\cal C} = \{H_1,
\ldots, H_s\}$ of size~$s$ there is a graph $G = G({\cal C})$ whose vertices
are the HT-classes $H_1, \ldots, H_s$ contained in ${\cal C}$, and where an
edge connects two distinct vertices $H_i, H_j$ if a matrix in $H_i$ can be
transformed to a matrix in $H_j$ by a single row/column switching operation,
possibly combined with transposition.

For example, it is known~\cite{ILL81,Kimura89} that there are $60$ H-classes of Hadamard
matrices of order $24$. These form two S-classes, of size $1$ and $59$.
Similarly, there are $36$ HT-classes, giving two ST-classes, of size $1$ and
$35$.  In each case the class of size $1$ contains the Paley matrix, which
has no closed quadruples. 

\section{Results for order $26$}		\label{sec:order26}

For order $26$ the
maximal determinant is $D(26) = 150 \times 6^{11} \times 2^{25},$
meeting the Ehlich-Wojtas bound~(\ref{eq:EWbound}), and the corresponding 
Gram matrix $G$ is unique up to symmetric signed permutations.
Without loss of generality we can assume that $G$ has a diagonal
block form with blocks of size
$13 \times 13$ (see~\cite{Ehlich64a,Orrick-www,Wojtas64}).
There are exactly three H-inequivalent maxdet matrices composed of circulant
blocks~\cite{Yang68,KKNK94}. However, there are many solutions that
are not composed of circulant blocks.
Orrick~\cite[Sec.~7]{Orrick08b} % arXiv0511141v2.pdf
found $5026$ HT-classes ($9884$ H-classes) % See size26/missed.txt
of solutions by a combination of hill-climbing (local optimisation)
and switching.

Using randomised decomposition of the Gram matrix $G$ followed by
switching, we have found $39$ further H-classes ($23$ HT-classes).
Thus, there are at least $9923$ H-classes ($5049$ HT-classes)
of saturated D-optimal designs of order $26$.
Since the randomised decomposition program has repeatedly found the same
set of $9923$ H-classes without finding any more, it is reasonable
to conjecture that this is all. An exhaustive search to prove this
may be feasible, but has not yet been attempted.

It is known~\cite{Ehlich64a,Orrick08b} that there are two
\emph{types} of maxdet matrices of order $n=26$, related to the two ways
that the row sums $2n-2=50$ of the Gram matrix 
can be written as a sum of squares:
\[50 = 7^2 + 1^2 = 5^2 + 5^2.\] They are called ``type $(7,1)$''
and ``type $(5,5)$'' respectively.
The type is preserved by switching. % Is this a theorem or just an observation?
If a maxdet
matrix $R$ of order $26$ is normalised so that $RR^T = R^TR = G$,
then 
\[\lambda(R) := \sum_i \big|\sum_j r_{i,j}\big| \]
determines the type of $R$: maxdet matrices of type $(7,1)$ 
have $\lambda(R) = 182$,  
and those of type $(5,5)$ have $\lambda(R) = 130$.

There are $5049$ HT-classes ($9923$ H-classes) 
which lie in $18$ ST-classes ($25$ S-classes).
There is one ``giant'' ST-class ${\cal G}$ with 
size $||{\cal G}||_{ST} = 4323$, consisting of type $(5,5)$ matrices.

There is another ``large'' ST-class ${\cal E}$ with 
$||{\cal E}||_{ST} = 686$, consisting of type $(7,1)$ matrices.

\begin{table}
\centering
\label{tab:table1}
\begin{tabular}{|c|c|c|c|c|c|}
\hline
$||{\cal C}||_S$ &$||{\cal C}||_{ST}$  & $w({\cal C})$ & type & splits & notes\\
\hline
8545	& 4323	& 229955/52	& $(5,5)$& no & ${\cal G} =\, <\!R_3\!>$\\ % Q8545, ST4323, <R3>, (5,5)
7	& 4	& 9/2		& $(5,5)$& no & new	\\	% Q7, ST4b, (5,5)
4	& 3	& 3/2		& $(5,5)$& no &	\\	% Q4, ST3b, (5,5)
1	& 1	& 1/2		& $(5,5)$& no &	\\	% Q1b, ST1f, (5,5)
1	& 1	& 1/6		& $(5,5)$& no & new	\\	% Q1d, ST1h, (5,5)
1	& 1	& 1/78		& $(5,5)$& no &	\\ 	% Q1e, ST1i, (5,5)
5, 5	& 5	& 11/6, 11/6	& $(5,5)$& yes &	\\	% Q10, ST5, (5,5)
4, 4	& 4	& 2, 2		& $(5,5)$& yes &	\\	% Q8, ST4a, (5,5)
1, 1	& 1	& 1/2, 1/2	& $(5,5)$& yes & new	\\	% Q2c, ST1c, (5,5)
1, 1	& 1	& 1/3, 1/3	& $(5,5)$& yes & new	\\ 	% Q2d, ST1d, (5,5)
1, 1	& 1	& 1/6, 1/6	& $(5,5)$& yes &	\\	% Q2b, ST1b, (5,5)
1310	& 686	& 3046/3	& $(7,1)$& no &	${\cal E}$ \\	% Q1310, ST686, (7,1)
19	& 11	& 6		& $(7,1)$& no & new	\\	% Q19, ST11, (7,1)
1	& 1	& 1/3		& $(7,1)$& no & new	\\	% Q1c, ST1g, (7,1)
1	& 1	& 1/39		& $(7,1)$& no & new	\\ 	% Q1f, ST1j, (7,1)
1	& 1	& 1/78		& $(7,1)$& no & $<\!R_2\!>$\\ % Q1a, ST1e, <R2>, (7,1)
3, 3	& 3	& 2/3, 2/3	& $(7,1)$& yes & new	\\	% Q6, ST3a, (7,1)
1, 1	& 1	& 1/78, 1/78	& $(7,1)$& yes & $<\!R_1\!>$\\ % Q2a, ST1a, <R1>, (7,1)
\hline
9923	& 5049	& 852013/156	& ---	 & --- & totals \\
\hline
\end{tabular}
\caption{25 S-classes and 18 ST-classes for order 26}
\end{table}

Each ST-class $\cal C$ of size $s = ||{\cal C}||_{ST}$  
corresponds to either 
one S-class ${\cal C}_1$ (of size $||{\cal C}_1||_S < 2s$) 
or two S-classes ${\cal C}_1$, ${\cal C}_2$ 
(each of size $||{\cal C}_i||_S = s$),
depending on whether or not the ST-class contains a self-dual matrix.
In the former case we say that the ST-class is {\em self-dual},
otherwise we say that the ST-class {\em splits}.
For example, the ST-class of size $11$ % ST11
is self-dual and
corresponds to an S-class of size $19$, %Q19
but the ST-class of size $5$ % ST5
splits to give two S-classes of size $5$. % Q10 (splits)
Details of all the known classes are given in Table~1. 
The third column of the table gives the weight(s) of the
S-class(es) in that row, where the weight is defined by
(\ref{eq:weight}) above. 
The entries labelled ``new'' are not given in
Orrick's paper~\cite{Orrick08b}.
The classes labelled $<\!R_i\!>$ ($1 \le i \le 3$),
are generated by matrices composed
of circulant blocks, using the notation of~\cite{Orrick-www}.
The graph associated with the ST class of size 11 is shown
in Figure~1.

The largest automorphism group
order is $22464 = 2^6\cdot 3^3\cdot 13$, and all group
orders divide $22464$. The distribution of group orders
is given in Table~2. In the table, the columns headed
``\#'' give the number of times that the corresponding group order
occurs. A list of (representatives of) H-classes and 
their group orders is available from~\cite{Brent-www}.

To summarise the main results, we have:

\begin{figure} \label{fig:1}
\begin{center}
\includegraphics[width=8cm]{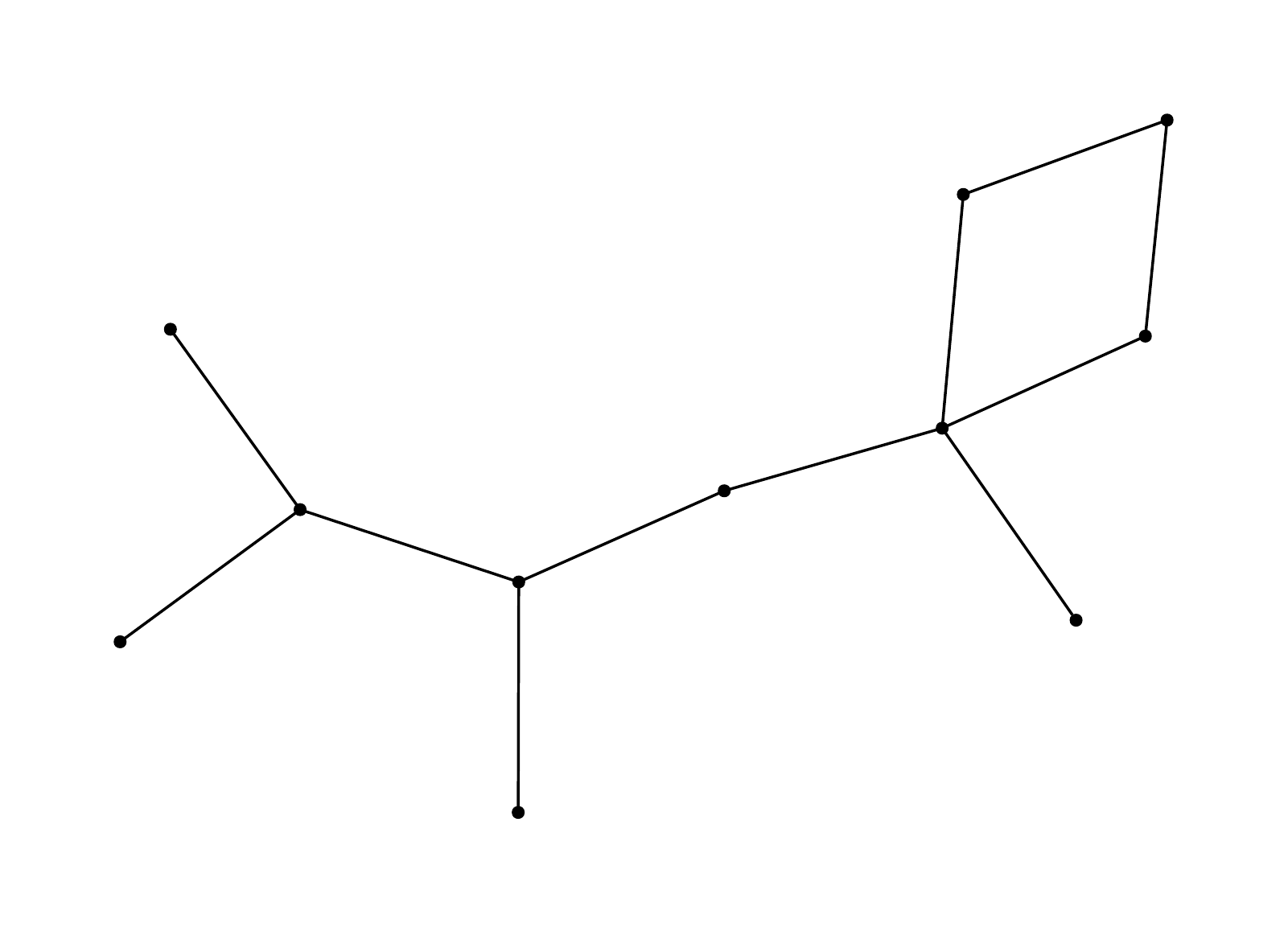} 
\end{center}
\caption{The ST-class of size 11 for maxdet matrices of order 26}
\end{figure}

\begin{thm}
For order $26$ there are at least $9923$ Hadamard classes of
$\{+1,-1\}$ matrices
with determinant $150 \times 6^{11} \times 2^{25}$.
They lie in at least $25$ switching classes, as given in
Table~$1$.	% \ref{tab:table1} does not work, why?
\end{thm}

\begin{proof}
The proof is computational. 
On our website~\cite{Brent-www} we give representatives
of each of the $18$ ST-classes. From these ``generators'', a program
that implements switching can find all $5049$ HT-classes;
this requires only 12 iterations of row/column switching and taking duals.
By taking duals of the $7$ generators that are not self-dual, we obtain
$25$ generators for the $9923$ H-classes.
\end{proof}

\begin{table}
\centering
\label{tab:table2}
\begin{tabular}{|c|c||c|c||c|c||c|c|}
\hline
Order & \# & Order & \# & Order & \# & Order & \# \\
\hline
1 & 2823 &
2 & 4086 &
3 & 41 &
4 & 1840\\
6 & 151 &
8 & 607 &
12 & 106&
16 & 143\\
24 & 20&
32 & 44&
36 & 6&
39 & 1\\
48 & 13&
64 & 13&
72 & 8&
78 & 6\\
96 & 3&
108 & 1&
156 & 2&
216 & 2\\
288 & 4&
576 & 2&
22464 & 1& & \\
\hline
\end{tabular}
\caption{Group orders of $9923$ H-classes for order 26}
\end{table}

\section{Results for order $27$}		\label{sec:order27}

It is known that the maximal determinant $D(27)$ for order $27$ satisfies
\[546 \le \frac{D(27)}{6^{11}\times 2^{26}} < 
  565,\]
where the lower bound is due to Tamura~\cite{Tamura05},
and the upper bound is the (rounded up) Ehlich bound~\cite{Ehlich64b}.
It is plausible to conjecture that the lower bound 
$546 \times 6^{11} \times 2^{26}$ is maximal,
since it is $0.9673$ of the Ehlich bound and has not
been improved despite attempts using optimisation
techniques that have been successful for other orders~\cite{Orrick-www}.
Unfortunately, proving that the lower bound is maximal seems
difficult~-- the technique used in~\cite{rpb244} to prove analogous
results for orders $19$ and $37$ is impractical for order $27$ due
to the size of the search space.

Tamura found a $\{+1,-1\}$ matrix $R$, with determinant
$546\times 6^{11}\times 2^{26}$.
The corresponding Gram matrix $G=RR^T$ has a block form with diagonal blocks
of sizes $(7,7,7,6)$. % see~\cite{Orrick-www}.
Orrick~\cite{Orrick08b} showed that Tamura's matrix $R$ generates 
an ST-class ${\cal T}$ with
$||{\cal T}||_{ST} = 33$.
The ST-class ${\cal T}$ splits into two S-classes, each containing
33 H-classes.

Using randomised decomposition of Tamura's (conjectured maximal) Gram
matrix, followed by switching, we have the following result.

\begin{thm}
There are at least $6489$ HT-classes $(12911$ H-classes$)$ of
$\{\pm 1\}$ matrices of order 27 
with determinant $546\times 6^{11} \times 2^{26}$.
They lie in at least $204$ ST-classes $(388$ S-classes$)$. % Not 398, see notes 
The largest ST-class contains at least $5765$ HT-classes 
$(11483$ H-classes$)$.
\end{thm}

\begin{proof}
The proof is computational. 
On our website~\cite{Brent-www} we give representatives
of each of the $204$ ST-classes. From these ``generators'', a program
that implements switching can find all $6489$ HT-classes;
this requires only 28 iterations of row/column switching and taking duals.
By taking duals of the $184$ generators that are not self-dual, we obtain
$388$ generators for the $12911$ H-classes.
\end{proof}

Details of the $204$ known ST-classes are summarised in Table~3.
Twenty of these ST-classes are self-dual; the remaining $184$ each split into
two S-classes. In the table, the columns headed 
``$||{\cal C}||_{ST}$'' give the size
of an ST-class $\cal C$, the next columns ``\#''
give the number of such classes, and the columns ``\#split'' 
give the number of these that split into two S-classes.

\begin{table}
\centering
\label{tab:table3}
\begin{tabular}{|c|c|c||c|c|c||c|c|c|}
\hline
$||{\cal C}||_{ST}$  & \# & \#split &
$||{\cal C}||_{ST}$  & \# & \#split &
$||{\cal C}||_{ST}$  & \# & \#split \\
\hline
5765	& 1 & 0 & 
36	& 1 & 1& 
33	& 1 & 1\\
28	& 1 & 1&
21	& 1 & 1&
18	& 2 & 2\\
14	& 2 & 2&
12	& 4 & 4&
11	& 1 & 1\\
9	& 2 & 2&
8	& 3 & 3&
7	& 7 & 5\\
6	& 12 & 12&
5	& 11 & 9&
4	& 12 & 11\\
3	& 18 & 17&
2	& 38 & 33&
1	& 87 & 79\\
\hline
\end{tabular}
\caption{204 ST-classes for order 27}
\end{table}

There is one ``{giant}'' class ${\cal G}$ of size 5765
HT-classes (11483 H-classes)
and 203 small classes (maximum size 36 HT-classes).
Tamura's matrix $R$ generates the third-largest class, 
of size 33 HT-classes (66 H-classes).
Unlike order 26 (see \S\ref{sec:order26}), there is no obvious subdivision of
the classes into types.

Automorphism group orders for the 12911 H-classes are summarised
in Table~4. Tamura's matrix $R$ has group order~$3$.

\begin{table}[h]
\centering
\label{tab:table4}
\begin{tabular}{|c|c|}
\hline
Order & multiplicity\\
\hline
1 & 12738\\
2 & 26 \\
3 & 131 \\
6 & 16\\
\hline
\end{tabular}
\caption{Distribution of group orders for $12911$ H-classes}
\end{table}

\begin{figure} \label{fig:2}
\begin{center}
\includegraphics[width=6cm]{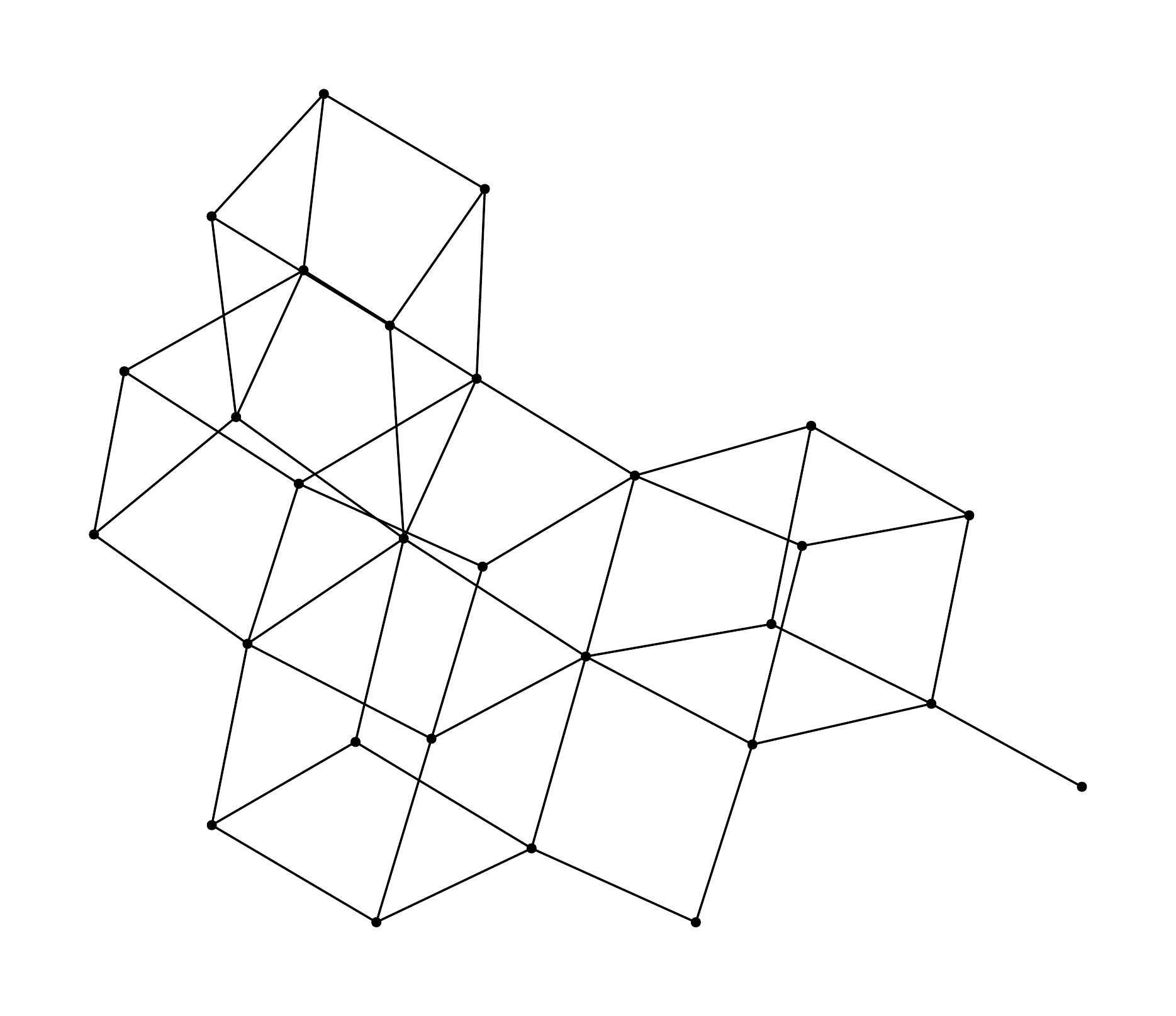}
\includegraphics[width=6cm]{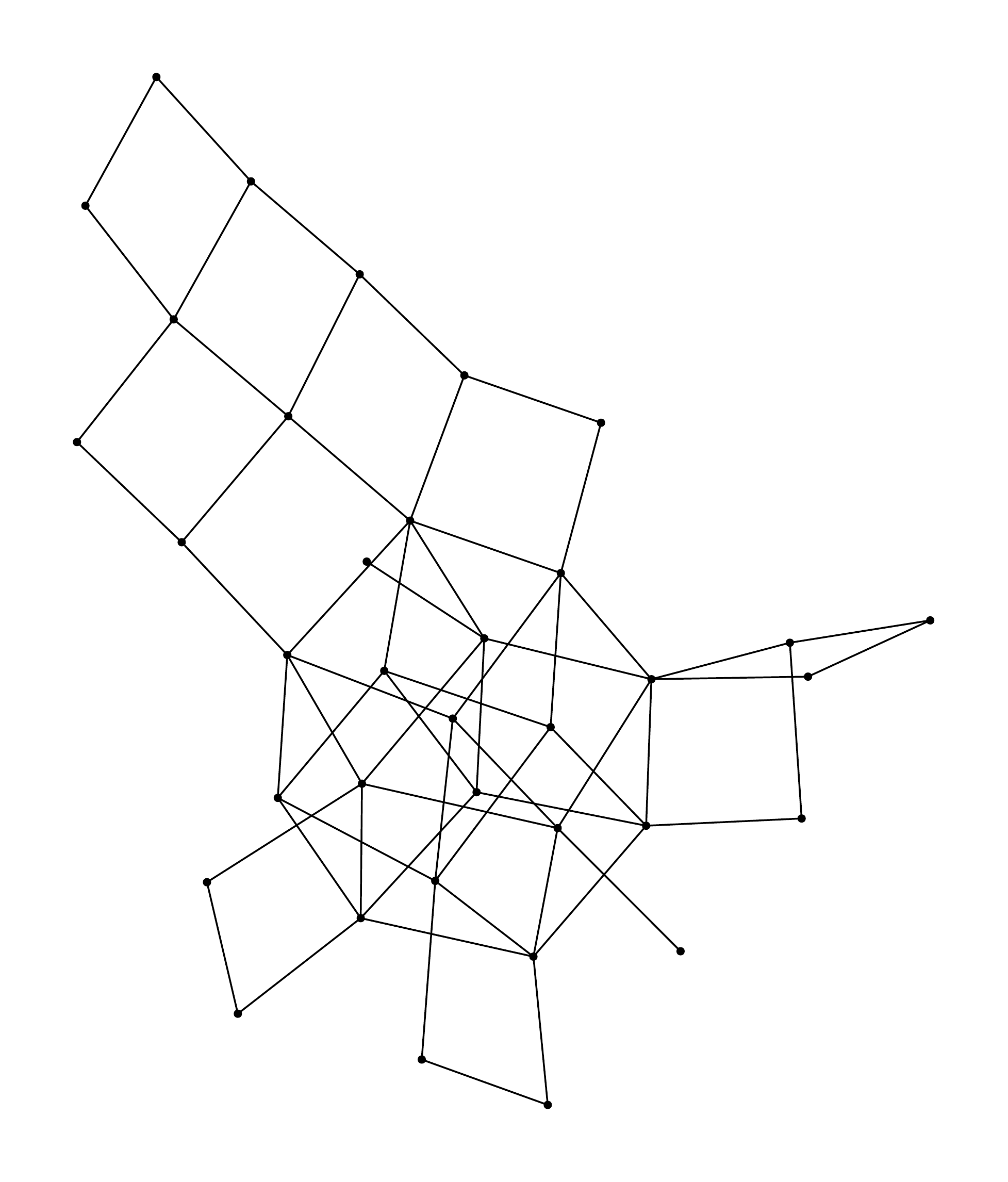}
\end{center}
\caption{ST-classes of size 28 and 36 for matrices of order 27}
\end{figure}

\section{Results for order $33$}		\label{sec:order33}

For order $33$, the Barba bound~(\ref{eq:Barbabound}) gives
$D(33) < 516\times 2^{74}$. Using the algorithm described
in~\cite{rpb244}, we have shown that none of the $13670$ candidate 
Gram matrices $G$ satisfying
$\det(G)^{1/2} \ge 470 \times 2^{74}$
can decompose into a product $RR^T$,
where $R \in \{+1,-1\}^{33 \times 33}$. Thus, we have
$D(33) < 470 \times 2^{74}$.

On the other hand, Solomon~\cite{Solomon02,Orrick-www}
found a matrix $R \in \{+1,-1\}^{33 \times 33}$ with 
$\det(R) = 441 \times 2^{74}$,
which is greater than $0.9382$ of the upper bound.
Thus, we know that
\[441 \le D(33)/2^{74} < 470.\]

It is plausible to conjecture that the lower bound
is best possible and $D(33) = 441 \times 2^{74}$.
Proving this seems difficult, for reasons given
in~\cite[\S7.1]{rpb244}.
In this section we
find a large number of H-classes of solutions
with determinant $441 \times 2^{74}$. 
Even if the conjecture proves to be incorrect,
the same techniques should be applicable to find many or all H-classes
of solutions with larger determinant.

Starting from the Gram matrix $G = R^TR = RR^T$ corresponding to
Solomon's $\{+1,-1\}$ matrix $R$, 
our randomised tree search algorithm
can find many solutions with the same determinant.

Using row/column switching and duality, we can find a huge number of
inequivalent solutions. For example, starting from Solonon's matrix $R$ and
iterating the operation of row switching only, we found $37030740$ H-classes
in 11 iterations before stopping our program because it was using too much
memory. Clearly a different strategy is needed.

\subsection*{Exploring the switching graph using random walks}

Given solutions $A_0$ and $B_0$, we can generate random walks
$(A_0,A_1,A_2,\ldots)$ and $(B_0,B_1,B_2,\ldots)$ in the graph
defined by row/column switching and transposition. Each vertex on a walk is
connected by a row/column switching operation and possibly transposition
to its successor.

If $A_0$ and $B_0$ are in different connected components, then the two
random walks can\,not intersect.  However, if $A_0$ and $B_0$ are in the
same connected component, of size $s$ say, then we expect the two walks to
intersect eventually, and probably after $O(\sqrt{s})$ steps unless the
mixing time of the walks is too long (this depends on the geometry of the
component, which is unknown).

Our implementation uses self-avoiding random walks. Each walk is stored in
a hash table so we can quickly check if a new vertex has already been 
encountered in the same walk (in which case we try one of its neighbours)
or in the other walk (in which case we have found an intersection).
If, during a walk, all neighbours of the current vertex have been visited,
then it is necessary to backtrack. This occurs rarely since the mean degree
of a vertex is large (see below).

We fix $A_0 = R$ and choose $B_0$ randomly. Usually (about 90\% of the time)
$R$ and $B_0$ are in the same connected component, nearly always the
``giant'' component ${\cal G}$.  Otherwise, $B_0$ is in a ``small''
component (of size say~$s$) and we discover this by being unable to continue
the self-avoiding walk from $B_0$ past $B_{s-1}$.

In this way we find many vertices of the giant component ${\cal G}$, 
and also many ``small'' ST-classes.

We can gather statistics from the random walks.  For example, we would like
to estimate $||{\cal G}||_{ST}$, the total number of
HT-classes (i.e.~connected components) in the graph, 
the number of ST-classes, the mean degree of each vertex, etc.

If implemented as
described above, the random walks are not uniform over the vertices of
the connected components containing their starting points.
They are approximately uniform over \emph{edges}, so the probability of 
hitting a vertex $v$ depends on the degree $\deg(v)$.

We can either take this into account when gathering statistics,
or avoid the problem by accepting a candidate vertex $v$
with probability {$1/\deg(v)$}. In this way
the vertices are sampled uniformly if the walks are long enough.
The drawback is that we have to compute the degrees of all candidate
vertices (all neighbours of the current vertex), which might be
time-consuming.

\subsection*{Results from random walks}

We estimate that the overall size of the graph is $(3.08\pm 0.09) \times 10^9$
when measured in HT-classes.  (In terms of H-classes the numbers
are roughly doubled, since self-dual classes are rare.)
The giant component ${\cal G}$ has size
$||{\cal G}||_{ST} \approx (2.83\pm 0.08) \times 10^9$.
In ${\cal G}$ the mean degree of each vertex is about $20$, so there
are about $2.83 \times 10^{10}$ edges.

We estimate that there are about $5 \times 10^7$ small ST-classes,
with mean size about 5.  Of these we found more than $8\times 10^4$ so far,
with the largest having size $2136$ (see Table~5).

We have found $1639$ singletons (ST-classes of size~1)
and $5412$ self-dual matrices. One self-dual singleton was found.

The automorphism group orders observed during random walks are in $\{1,2,4\}$,
with orders greater than~$1$ being rare. Solomon's $R$ has group order~$2$.

Although most of these observations are imprecise, since they depend on
random sampling, we can at least claim the following:

\begin{thm}
For order $33$ and determinant $441\times 2^{74}$, 
the ST-class ${\cal G}$ generated by Solomon's matrix $R$ is self-dual
and has size $||{\cal G}||_{ST} > 197\times 10^6$.
There are at least $8\times 10^4$ smaller ST-classes, $20$ of
which are listed in Table~$5$.
\end{thm}

\begin{proof}
As usual, the proof is computational.
Starting from $R$, we found $197122852$ HT-classes
in $9$ iterations of row/column switching and 
taking duals.  

Starting random walks
from $R$ and $R^T$, and performing row/column switching only, we found an
intersection. % after $912703$ iterations, using just row switching
Thus $R \leftrightarrow R^T$.
It follows that $\cal G$ is self-dual. % Should this be a lemma ?
\end{proof}

\subsection*{Remarks}

{\hspace*{1.5em}\bf 1.} Although $R$ is not self-dual, we found many self-dual
matrices in the giant class $\cal G$ in the course of various random walks.
Eight such matrices are at distance $3$ from $R$. The existence of a
self-dual matrix in $\cal G$ is sufficient to show that $\cal G$ is
self-dual.

\pagebreak[3]
\begin{table}[ht]	\label{tab:table5}
\begin{center}
\begin{tabular}{|c|c||c|c|}
\hline
size $s_i$ & $\lambda_i$ & size $s_i$ & $\lambda_i$\\
\hline
2136 	& 2  	& 1100 & 1\\
1300   	& 4	& 1069 & 2\\
1276   	& 2	& 1011 & 1\\	% Only one class of size 1300
1246    & 1	& 1008 & 1\\	% Class of size 1246 found 20111231
1205	& 4	& 999 & 1(a)\\	% Two distinct classes of size 999
1188	& 4	& 999 & 2(b)\\
1187	& 1	& 993 & 2\\
1148	& 2	& 958 & 2\\
1134	& 2	& 918 & 3\\
1104	& 2	& 909 & 3\\
\hline
\end{tabular}
\caption{Some large ST-classes for order $33$}
\end{center}
\end{table}

{\bf 2.} In addition to the giant class ${\cal G}$ 
of size about $2.83\times 10^9$,
we found $20$ ST-classes of size $\ge 900$, as listed in Table~5,
with the largest having size $2136$.
Classes of the same size marked ``(a)'' and ``(b)'' are different,
so there are (at least) two disjoint classes of size $999$.

{\bf 3.} Table 5 gives the number of times $\lambda_i$ 
that we found the same class of size~$s_i$.
This statistic is mentioned because it indicates how well we have sampled 
the search space. Excluding the giant
class $\cal G$, the $20$ largest known classes, 
of total size $\sum s_i = 22888$,
were found $\sum \lambda_i = 42$ times. 
Consider the union $U$ of these classes as a sample from the space,
and {\em assume} that the space is sampled uniformly.
Excluding  the $20$ ``hits'' used to select $U$, there are 
$\sum (\lambda_i-1) = 22$ additional
``hits'' on $U$. Thus, the fraction $\rho$ of the space sampled is 
$\rho \approx \sum (\lambda_i-1)/\sum s_i = 22/22888 \approx 1/1040$.
Under our assumption, the probability $P$
of missing a given class of size $\ge  2136$ is bounded by
\begin{equation}
P \le \left(1 - \rho\right)^{2136} < \; {1}/{7}\,.
\label{eq:Pest}
\end{equation}
On the other hand, random sampling hit the giant class
$1.04\times 10^6$ times, and the estimated size of this 
class is $2.83\times 10^9$, implying that $\rho \approx 1/2720$, 
so the estimate~(\ref{eq:Pest}) on $P$ should be viewed with caution.
The discrepancy between the two estimates of $\rho$ 
may be caused by nonuniformity of sampling and/or by an
inaccurate estimate of the size of the giant class.

{\bf 4.} It would be interesting to know more about the graphs
associated with ``small'' ST classes.
We have observed one graph of size~3 (``$\vee$''),
two of size~4 (``$\sqcup$'' and ``$\square$''), and only one
of size~5 (``kite'').
Figure~3 shows an example of each size in the range $10,\ldots,19$.

\pagebreak[3]

{\bf 5.} The reader may have noticed that the graphs displayed in Figures
1--3 are bipartite ($2$-colourable), although \emph{neato} did not draw them
in a way that makes this obvious. Computational experiments have shown that
most, but not all, of the ST classes considered above
have bipartite graphs. In particular,
the graph of the giant component for order $33$ is not bipartite.

\pagebreak[3]
\subsection*{Acknowledgements}

The author thanks Judy-anne Osborn for introducing him to the topic and for
many stimulating discussions; Will Orrick and Paul Zimmermann for their
ongoing collaboration; Brendan McKay for his graph isomorphism program
\emph{nauty}~\cite{McKay09}, which was used to check Hadamard equivalence; 
AT\&T for the \emph{Graphviz} graph visualization software, in particular
\emph{neato}, which was used to draw the figures; and the Mathematical
Sciences Institute (ANU) for computer time on the cluster ``orac''.

\vspace*{\fill}
\begin{figure}[hb] \label{fig:3}
\begin{center}
\includegraphics[width=10cm]{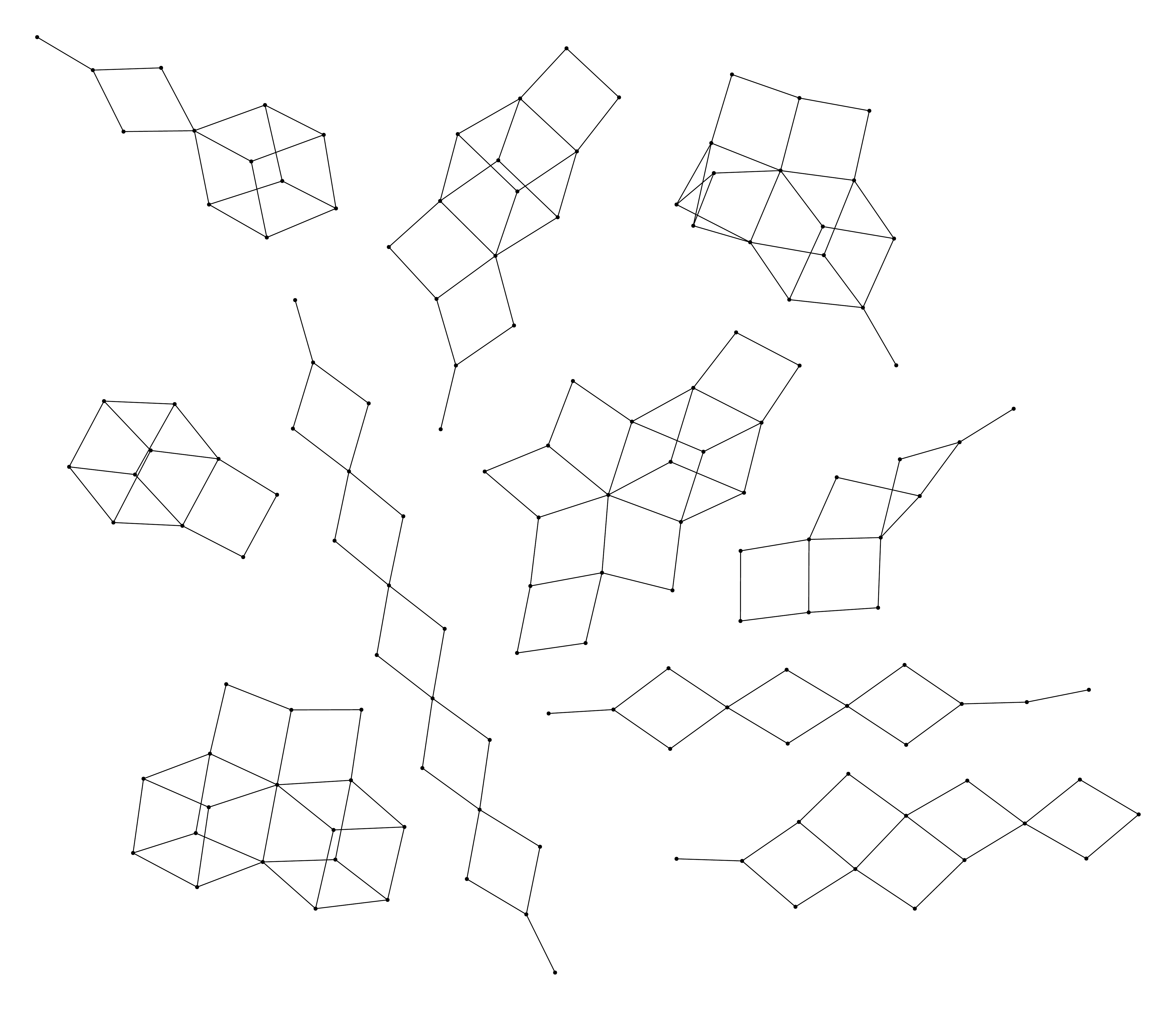}
\end{center}
\caption{One ST-class of each size $10,\ldots,19$ for order $33$}
\end{figure}
\vspace*{\fill}

\pagebreak[4]		% Does not waste any pages and looks better


\begin{thebibliography}{99}

\bibitem{Barba33}
G. Barba, Intorno al teorema di Hadamard sui determinanti a valore massimo,
{\em Giorn. Mat. Battaglini} {\bf{71}} (1933), 70--86.

\bibitem{Brent-www}
R. P. Brent, \textsl{The Hadamard maximal determinant problem},
\url{http://maths.anu.edu.au/~brent/maxdet/}

\bibitem{rpb244}
R. P. Brent, W. P. Orrick, J. H. Osborn and P. Zimmermann,
Maximal determinants and saturated D-optimal designs of orders $19$ and $37$,
submitted.		% Submitted to DCC on 20111219
Available from \url{http://arxiv.org/abs/1112.4160}.

\bibitem{rpb249}
R. P. Brent and J. H. Osborn,
General lower bounds on maximal determinants of binary matrices,
submitted. 		% Submitted to EJC 20120809
Available from \url{http://arxiv.org/abs/1208.1805}.

\bibitem{Brouwer83}
A. E. Brouwer, {\em An infinite series of symmetric designs}, 
Math. Centrum, Amsterdam, Report ZW 202/83 (1983).

\bibitem{Denniston82} 
R. H. F. Denniston,
Enumeration of symmetric designs $(25,9,3)$. 
In {\em Algebraic and geometric combinatorics}, 
Volume 65 of North-Holland Math. Stud.,
North-Holland, Amsterdam, 1982, 111--127.
 
\bibitem{Ehlich64a}
H. Ehlich, Determinantenabsch\"atzungen f\"ur bin\"are Matrizen,
{\em Math. Z.} {\bf{83}} (1964), 123--132.

\bibitem{Ehlich64b}
H. Ehlich, Determinantenabsch\"atzungen f\"ur bin\"are Matrizen mit 
$N \equiv 3 \bmod 4$,
{\em Math. Z.} {\bf{84}} (1964), 438--447.

\bibitem{Hadamard93}
J. Hadamard,
R\'esolution d'une question relative aux d\'eterminants,
\emph{Bull. des Sci. Math.} \textbf{17} (1893), 240--246.

\bibitem{ILL81}
N. Ito, J. S. Leon and J. Q. Longyear, 
Classification of 3-(24,12,5) designs and 24-dimensional Hadamard matrices, 
{\em J.\ Combin.\ Theory Ser.\ A} {\bf 31} (1981), 66--93.

\bibitem{Kimura89}
H. Kimura, 
New Hadamard matrix of order 24, 
{\em Graphs Combin.\ }{\bf 5} (1989), 235--242.

\bibitem{KKS91}
C. Koukouvinos, S. Kounias and J. Seberry,
Supplementary difference sets and optimal designs,
{\em Discrete Math.} {\bf 88} (1991), 49--58.

\bibitem{KKNK94}
S. Kounias, C. Koukouvinos, N. Nikolaou and A. Kakos, 
The non-equivalent circulant D-optimal designs for $n \equiv 2$ mod $4$, 
$n \le 54$, $n=66$, 
{\em J.\ Combin.\ Theory Ser.\ A} {\bf 65} (1994), 26--38.

\bibitem{McKay79}
B. D. McKay,
Hadamard equivalence via graph isomorphism,
\emph{Discrete Mathematics} \textbf{27} (1979), 213--214.

\bibitem{McKay09}
B. D. McKay,
{\em nauty} User's Guide (Version 2.4),
Dept.\ of Computer Science,
Australian National University,
November 2009.

\bibitem{Orrick08a}
William P. Orrick,
Switching operations for Hadamard matrices,
{\em SIAM J.\ Discrete Math.\ }{\bf 22} (2008), 31--50.

\bibitem{Orrick08b}
William P. Orrick, 
On the enumeration of some D-optimal designs,	
{\em J.~Statist. Plann. Inference} {\bf{138}} (2008) 286--293.
\url{http://arxiv.org/abs/math/0511141v2}

\bibitem{Orrick-www} 
William P. Orrick,
\textsl{The Hadamard maximal determinant problem},
\url{http://www.indiana.edu/~maxdet/}

\bibitem{Osborn02}
Judy-anne H. Osborn, 
\textsl{The Hadamard Maximal Determinant Problem},
Honours Thesis, University of Melbourne, 2002, 142~pp.
\url{http://wwwmaths.anu.edu.au/~osborn/publications/pubsall.html}

\bibitem{Paley33}
R. E. A. C. Paley,
On orthogonal matrices,
\emph{J. Math. Phys} \textbf{12} (1933), 311--320.

\bibitem{Solomon02}
B. Solomon,
Personal communication to W.~Orrick, 18 June 2002.

\bibitem{Spence75}
E. Spence,
Skew-Hadamard matrices of the Goethals-Siedel type,
{\em Canadian J.\ Math.} {\bf 27} (1975), 555--560.

\bibitem{Tamura05}
H. Tamura, % Hiroki Tamura
Personal communication to W.~Orrick, 26 August 2005.

\bibitem{Wanless04}
Ian M.\ Wanless, 
Cycle switches in latin squares,
{\em Graphs Combin.\ }{\bf 20} (2004), 545--570.

\bibitem{Whiteman90}
A. L. Whiteman,
A family of D-optimal designs,
{\em Ars Combinatoria} {\bf 30} (1990), 23--26.

\bibitem{Wojtas64}
W. Wojtas,
On Hadamard's inequality for the determinants of order non-divisible by $4$,
\emph{Colloq. Math.} \textbf{12} (1964), 73--83.

\bibitem{Yang68}
C. H. Yang, 
On designs of maximal $(+1, -1)$-matrices of order $n \equiv 2$ (mod $4$),
{\em Math.\ Comp.\ } {\bf 22} (1968), 174--180.

\end{thebibliography}
\end{document}